\documentclass[12pt]{amsart}
\usepackage{amssymb, amsmath}
\title{ Dehn functions and the space of marked groups}
\author{M. Shahryari}
\address{M. Shahryari: Department of Pure Mathematics,  Faculty of Mathematical
Sciences, University of Tabriz, Tabriz, Iran}
\email{mshahryari@tabrizu.ac.ir}

\newtheorem {theorem}{Theorem}

\begin{document}

\maketitle
\begin{abstract}
In the space of marked group, we suppose that a sequence $(G_i, X_i)$ converges to $(G,X)$, where $G$ is finitely presented. We obtain an inequality which connects Dehn functions of $G_i$s and $G$. 
\end{abstract}


\vspace{1cm}

In the space of marked groups, consider the situation a sequence $(G_i, X_i)$  in which  converges to a finitely presented marked group $(G, X)$. What can we say about the relation between corresponding Dehn functions of the groups $G_i$ and $G$? Suppose $\Gamma_i=\langle X_i|R_i\rangle$ is an arbitrary presentation for $G_i$ and $\Gamma=\langle X|R\rangle$ is an arbitrary finite presentation for $G$. Let $L=\max_{r\in R}\|r\|$. We will prove that  for any $n\geq 0$,
$$
\limsup_i\frac{\delta_i(n)}{\delta_i(L)}\leq \delta(n).
$$
Here of course, $\delta_i$ is the Dehn function of $G_i$ corresponding to $\Gamma_i$. Also $\delta$ is the Dehn function of $G$ corresponding to a finite presentation $\Gamma$. As a result, it shows that if the set $\{ \delta_i(L):\ i\geq 1\}$ is finite, then so is the set $\{ \delta_i(n):\ i\geq 1\}$, for all $n\geq 0$.

\section{Basic notions}
The idea of Gromov-Grigorchuk metric on the space of finitely generated groups is proposed by M. Gromov in his celebrated solution to the Milnor's conjecture on the groups with polynomial growth (see \cite{Gromov}). For a detailed discussion of this metric, the reader can consult \cite{Cham}. Here, we give some necessary basic definitions. A marked group $(G,X)$ consists of a group $G$ and an $m$-tuple of its elements $X=(x_1, \ldots, x_m)$ such that $G$ is generated by $X$. Two marked groups $(G, X)$ and $(G^{\prime}, X^{\prime})$ are {\em the same}, if there exists an isomorphism $G\to G^{\prime}$ sending any $x_i$ to $x^{\prime}_i$. The set of all such marked groups is denoted by $\mathcal{G}_m$. This set can be identified by the set of all normal subgroup of the free group $\mathbb{F}=\mathbb{F}_m$. Since the later is a closed subset of the compact topological space $2^{\mathbb{F}}$ (with the product topology), so it is also a compact space. This space is in fact metrizable: let $B_{\lambda}$ be the closed ball of radius $\lambda$  in $\mathbb{F}$ (having the identity as the center) with respect to its word metric. For any two normal subgroups $N$ and $N^{\prime}$, we say that they are in distance at most $e^{-{\lambda}}$, if $B_{\lambda}\cap N=B_{\lambda}\cap N^{\prime}$. So, if $\Lambda$ is the largest of such numbers, then  we can define
$$
d(N, N^{\prime})=e^{-\Lambda}.
$$
This induces a corresponding metric on $\mathcal{G}_m$. To see what is this metric exactly, let $(G, X)$ be a marked group. For any non-negative integer $\lambda$, consider the set of {\em relations} of $G$ with length at most $\lambda$, i.e.
$$
\mathrm{Rel}_{\lambda}(G, X)=\{ w\in \mathbb{F}: \| w\|\leq \lambda, w_G=1\}.
$$
Then $d((G, X),(G^{\prime}, X^{\prime}))=e^{-\Lambda}$, where $\Lambda$ is the largest number such that $\mathrm{Rel}_{\Lambda}(G,X)=\mathrm{Rel}_{\Lambda}(G^{\prime}, X^{\prime})$. Note that we identify here $X$ and $X^{\prime}$ using the correspondence $x_i\to x^{\prime}_i$. This metric on $\mathcal{G}_m$ is the so called Gromov-Grigorchuk metric.

Many topological properties of the space $\mathcal{G}_m$ is discussed in \cite{Cham}. In this note, we need just one elementary fact: any finitely presented marked group $(G, X)$ in $\mathcal{G}_m$ has a neighborhood, every element in which is a quotient of $G$. \\

We also need to review some basic notions concerning {\em Dehn}  and {\em isoperimtry}  functions. Let $\langle X|R\rangle$ be a presentation for a finitely generated group $G$ ($X$ is finite).  Let $w\in \mathbb{F}=F(X)$ be a word such that $w_G=1$. Clearly in this case $w$ belongs to $\langle R^{\mathbb{F}}\rangle$, the normal closure of $R$ in $\mathbb{F}$. Hence, we have
$$
w=\prod_{i=1}^ku_ir_i^{\pm 1}u_i^{-1},
$$
for some elements $u_1, \ldots, u_k\in \mathbb{F}$ and $r_1, \ldots, r_k\in R$. The smallest possible $k$ is called the {\em area} of $w$ and it is denoted by $\mathrm{Area}_R(w)$. A function $f:\mathbb{N}\to \mathbb{R}$ is an {\em isoperimetric} function for the given presentation, if for all $w\in \mathbb{F}$, with $w_G=1$, we have
$$
\mathrm{Area}_R(w)\leq f(\| w\|).
$$
The corresponding Dehn function is the smallest isoperimetric function, i.e.
$$
\delta(n)=\max \{ \mathrm{Area}_R(w): w\in \mathbb{F}, w_G=1, \|w\|\leq n\}.
$$
This function measures the complexity of the word problem in the case of finitely presented group $G$: the word problem for the presentation $\langle X|R\rangle$ is solvable, if and only if, the corresponding Dehn function is recursive. In fact the recursive Dehn functions measures the time complexity of fastest non-deterministic Turing machine solving word problem of $G$ (see \cite{Gres} and \cite{Sapir}).  Also in the case of finitely presented groups, the {\em type} of Dehn function is a quasi-isometry invariant of groups. Although, in this note, we use the exact values of Dehn function, we give the definition of {\em type}.  Let $f, g:\mathbb{N}\to \mathbb{N}$ be two arbitrary functions. We say that $f$ is {\em dominated} by $g$, if there exists a positive number $C$, such that for all $n$,
$$
f(n)\leq Cg(Cn+C)+Cn+C.
$$
We denote the domination by $f\preceq g$. These two functions are said to be equivalent, if $f\preceq g$ and $g\preceq f$. The type of a Dehn function is its equivalence class with respect to this relation. Two Dehn functions of a fixed group with respect to different finite presentations have the same type. There are many classes of finitely presented groups having Dehn functions of type $n^{\alpha}$ for a dense set of exponents $\alpha\geq 2$ (see \cite{Brad}). Hyperbolic groups are the only groups having linear type Dehn functions. Olshanskii proved that there is no group with Dehn function of type $n^{\alpha}$ with $1< \alpha <2$ (see \cite{Olsh}). For a study of Dehn functions of non-finitely presented groups, the reader can consult \cite{Gri}.

\section{Main results}
We work within the space of marked groups $\mathcal{G}_m$.

\begin{theorem}
Let $(G_i, X_i)$ be a sequence converging to $(G, X)$, where $G$ is finitely presented. Then for any $n$, we have
$$
\limsup_i\frac{\delta_i(n)}{\delta_i(L)}\leq \delta(n),
$$
where $\delta_i$ is any Dehn function of $G_i$ and $\delta$ is the  Dehn function of $G$ corresponding to any finite presentation $\Gamma=\langle X|R\rangle$ and $L=\max_{r\in R}\|r\|$.
\end{theorem}

\begin{proof}
As $G$ is finitely presented, we may assume that all $G_i$ is a quotient of $G$. We also identify $X_i$ by $X$ using the obvious correspondence. Let $\mathbb{F}=F(X)$ be the free group on $X$ and assume that $w\in \mathbb{F}$. Suppose that $w_G=1$ and $l=\mathrm{Area}_R(w)$. Then we have
$$
w=\prod_{j=1}^l a_jr_j^{\pm 1}a_j^{-1},
$$
where $a_1, \ldots, a_l\in \mathbb{F}$ and $r_1, \ldots, r_l\in R$. We know that $(r_j)_{G_i}=1$, for all $i$ and $j$, hence
$$
r_j=\prod_{t_j=1}^{l_{ij}}u_{it_j}r_{it_j}^{\pm 1}u_{it_j}^{-1},
$$
where $l_{ij}=\mathrm{Area}_{R_i}(r_j)$,  $r_{i1}, \ldots, r_{il_{ij}}\in R_i$ and $u_{i1}, \ldots, u_{il_{ij}}\in \mathbb{F}$. Therefore, we have
\begin{eqnarray*}
w&=& \prod_{j=1}^l a_j(\prod_{t_j=1}^{l_{ij}}u_{it_j}r_{it_j}^{\pm 1}u_{it_j}^{-1})^{\pm 1}a_j^{-1}\\
 &=& \prod_{j=1}^l\prod_{t_j=1}^{l_{ij}}a_j u_{it_j}r_{it_j}^{\mp 1}u_{it_j}^{-1}a_j^{-1}.
\end{eqnarray*}
This shows that
$$
\mathrm{Area}_{R_i}(w)\leq \sum_{j=1}^ll_{ij}=\sum_{j=1}^l\mathrm{Area}_{R_i}(r_j).
$$
Suppose $K_i=\max_{r\in R}(\mathrm{Area}_{R_i}(r))$. Then, we have
$$
(\ast)\ \ \ \mathrm{Area}_{R_i}(w)\leq K_i\cdot \mathrm{Area}_R(w).
$$
Now, let $n\geq 1$. There exists an integer $i_0$ such that for any $i\geq i_0$, we have
$$
d((G_i, X_i), (G, X))\leq e^{-n}.
$$
This shows that  $\mathrm{Rel}_n(G_i, X_i)=\mathrm{Rel}_n(G, X)$, for $i\geq i_0$. In other words
$$
\{ w\in \mathbb{F}:\ \|w\|\leq n, w_{G_i}=1\}=\{w\in \mathbb{F}:\ \|w\|\leq n, w_G=1\}.
$$
By $(\ast)$ and by the definition of Dehn function, we conclude $\delta_i(n)\leq K_i\cdot \delta(n)$. Hence, for $i\geq i_0$, we have
$$
\frac{\delta_i(n)}{K_i}\leq \delta(n),
$$
and therefore
$$
\sup_{i\geq i_0}\frac{\delta_i(n)}{K_i}\leq \delta(n).
$$
For any $j$, define
$$
a_j(n)= \sup_{i\geq j}\frac{\delta_i(n)}{K_i}\leq \delta(n),
$$
which a decreasing sequence in $j$. Since $a_{i_0}(n)\leq \delta(n)$, so $\lim_j a_j(n)\leq \delta(n)$. This shows that
$$
\limsup_i\frac{\delta_i(n)}{K_i}\leq \delta(n).
$$
Now, note that
\begin{eqnarray*}
K_i&=&\max_{r\in R}\mathrm{Area}_{R_i}(r)\\
   &\leq&\max_{r\in R, \|r\|=\|w\|}\mathrm{Area}_{R_i}(w)\\
   &\leq&\delta_i(L).
\end{eqnarray*}
This completes the proof.
\end{proof}

As a result, we see that if the set $\{ \delta_i(L):\ i\geq 1\}$ is finite, then so is the set $\{ \delta_i(n):\ i\geq 1\}$, for all $n\geq 0$. This is because, if we put $M=\max_i\delta_i(L)$, then
$$
\limsup_i \delta_i(n)\leq M\cdot \delta(n).
$$
Now, if the second set is infinite, then the sequence $(\delta_i(n))_i$ has a divergent subsequence, which is a contradiction.


\begin{thebibliography}{99}

\bibitem{Brad}
Brady N., Bridson M. R.  {\it There is only one gap in the isoperimetric spectrum}. Geometric and Functional Analysis, 2000, {\bf 10 (5)}, pp. 1053-1070.

\bibitem{Cham} Champtier C., Guirardel V.
{\it Limit groups as limits of free groups}. Israel J. Math., 2004, {\bf 146(1)}, pp. 1-75.

\bibitem{Gres}
Gresten S. M. {\it Isoperimetric and isodiametric functions of finite presentations}. Geometric Group Theory, London Math. Soc. Lecture Notes, 1991, {\bf 181}, pp. 79-96.

\bibitem{Gri}
Grigorchuk R. I., Ivanov S. V.  {\it On Dehn Functions of Infinite Presentations of Groups}. Geometric and Functional Analysis, 2009, {\bf 18(6)}, pp. 1841-1874.

\bibitem{Gromov}
Gromov M. {\it Groups of polynomial growth and expanding maps}. Inst. Hautes Etudes Sci. Publ. Math., 1981, {\bf 53},
pp. 53-73.



\bibitem{Olsh}
Olshanskii A. Yu. {\it Hyperbolicity of groups with subquadratic isoperipetric inequality}.
International J. Algebra and Computations, 1991, {\bf 1 (3)}, pp. 281-289.




\bibitem{Sapir}
Sapir M., Birget J. C., Rips E.  {\it Isoperimetric and isodiametric functions of groups}.
Annals of Math., 2001, {\bf 181}. pp: 345-366.


\end{thebibliography}
\end{document}